\documentclass[10pt, english]{article}
\usepackage[utf8]{inputenc}
\usepackage{amssymb,amsthm,amsmath, amsfonts}
\usepackage{pst-all,pst-3dplot,pstricks,pstricks-add,pst-math,pst-xkey}
\usepackage{graphicx}
\usepackage{latexsym}
\usepackage{accents}
\usepackage{hyperref}
\usepackage{combgames}

\usepackage{dingbat}
\usepackage{stmaryrd}

\usepackage{bbold}
\usepackage{graphs}
\usepackage{enumerate}
\usepackage{epsfig}
\usepackage{subfigure}
\usepackage{skak}
\usepackage{chessboard}
\usepackage{chessfss}
\usepackage{lscape}

\newtheorem{theorem}{Theorem}
\newtheorem{proposition}[theorem]{Proposition}
\newtheorem{corollary}[theorem]{Corollary}
\newtheorem{lemma}[theorem]{Lemma}
\newtheorem*{theorem2}{Theorem 2}

\theoremstyle{definition}
\newtheorem{definition}{Definition}

\newtheorem{problem}{Problem}

\newtheorem{example}{Example}
\newcommand{\gL}{g^\mathcal{L}}
\newcommand{\gR}{g^\mathcal{R}}
\newcommand{\GL}{G^\mathcal{L}}
\newcommand{\GR}{G^\mathcal{R}}

\renewcommand{\L}{\ensuremath {\mathrm{L}}}
\newcommand{\R}{\ensuremath{\mathrm{R}}}
\newcommand{\N}{\ensuremath {\mathrm{N}}}
\renewcommand{\P}{\ensuremath{\mathrm{P}}}

\newcommand{\bl}{{\,\bullet}}
\newcommand{\wh}{{\,\circ}}

\newcommand{\aw}{\ensuremath{\mathrm{aw}}}

\newcommand{\cg}[2]{\{#1\mid#2\}}
\def\cgstar{\mathord{\ast}}
\def\cgup{\mathord{\uparrow}}

\def\up{\mathord{\uparrow}}
\def\star{\mathord{\ast}}
\def\cgdown{\mathord{\downarrow}}
\def\bip{{\sc bipass }}
\def\ub{{unit-bypass }}
\def\ubp{{unit-bypass}}
\def\nb{{neigbor-bypass }}

\renewcommand{\ge}{\geqslant}

\newcommand{\bipass}{{\sc bipass}}

\newcommand\upstarr{\mathrel{\rotatebox{90}{$\rightarrowtail$}}}
\newcommand\downstarr{\mathrel{\rotatebox{270}{$\rightarrowtail$}}}

\usepackage{newpxmath} 
\newcommand\upstar{\mkern0.5mu\lower0.75ex\hbox{$\scriptstyle\upstarr$}}

\newcommand\downstar{\mkern0.5mu\raise1.98ex\hbox{$\scriptstyle\downstarr$}}

\usepackage{array}
\usepackage{tikz}

\usetikzlibrary{decorations.pathreplacing,}
\tikzset{radiation/.style={{decorate,decoration={expanding waves,angle=90,segment length=4pt}}},
        amoebaL/.pic={
        		code={\tikzset{scale=5/10}
        		\draw[scale=0.75,fill=gray, rotate=-15] (0,0) ellipse (1 and {2});
  		\draw[fill=white] (.53,0.5) ellipse (.3 and {.15});
  		}
  	}
}

\tikzset{radiation/.style={{decorate,decoration={expanding waves,angle=90,segment length=4pt}}},
amoebaR/.pic={
        		code={\tikzset{scale=5/10}
        		\draw[scale=0.75, rotate=15] (0,0) ellipse (1 and {2});
  		\draw[fill=gray] (-0.55,0.5) ellipse (.25 and {.13});
  		}
  	}
}

\begin{document}

\title{Atomic weights and the combinatorial game of {\sc bipass}}
\author{Urban Larsson\footnote{National University of Singapore, urban031@gmail.com, partly supported by the Killam Trusts.} \and Richard J. Nowakowski\footnote{Dalhousie University, r.nowakowski@dal.ca}}
\maketitle

\begin{abstract}
We define an {\em all-small} ruleset, {\sc bipass}, within the framework of normal-play combinatorial games. A game is played on finite strips of black and white stones. Stones of different colors are swapped provided they do not bypass one of their own kind. We find a surjective function from the strips to integer {\em atomic weights} (Berlekamp, Conway and Guy 1982) that measures the number of units in all-small games. This result provides explicit winning strategies for many games, and in cases where it does not, it gives narrow bounds for the canonical form game values. We prove that the game value $*2$ does not appear as a disjunctive sum of {\sc bipass}. Moreover, we find game values for some parametrized families of games, including an infinite number of strips of value $*$.
\end{abstract}

\section{Introduction}
A bi-collective of one-directional micro organisms, consisting of a black tribe and a white tribe, live in close proximity, and they take turns  moving. 
A collective is divided into a finite number of one-dimensional units (strips), and this number cannot increase, because units cannot split. See Figure~\ref{fig:collective}.

\begin{figure}[ht!]
\begin{center}
\begin{tikzpicture}[scale=1, every node/.style={transform shape}, line width=1pt]
    \path (-2.6,0) pic {amoebaL};
    \path (-1.3,0) pic {amoebaL};
    \path (0,0) pic {amoebaL};
    \path (1.3,0) pic {amoebaR};
    \path (2.6,0) pic {amoebaL};
    \path (3.9,0) pic {amoebaR};
\end{tikzpicture}\vspace{4 mm}

\begin{tikzpicture}[scale=1, every node/.style={transform shape}, line width=1pt]
    \path (-1.3,0) pic {amoebaL};
    \path (0,0) pic {amoebaL};
    \path (1.3,0) pic {amoebaR};
    \path (2.6,0) pic {amoebaR};
    \path (3.9,0) pic {amoebaL};
\end{tikzpicture}\vspace{4 mm}

\begin{tikzpicture}[scale=1, every node/.style={transform shape}, line width=1pt]
    \path (0,0) pic {amoebaL};
    \path (1.3,0) pic {amoebaR};
    \path (2.6,0) pic {amoebaL};
    \path (3.9,0) pic {amoebaR};
\end{tikzpicture}\vspace{4 mm}

\begin{tikzpicture}[scale=1, every node/.style={transform shape}, line width=1pt]
    \path (5.2,0) pic {amoebaR};
    \path (2.6,0) pic {amoebaL};
    \path (3.9,0) pic {amoebaR};
\end{tikzpicture}\caption{A bi-collective (a.k.a. disjunctive sum) of 4 amoebae units. All but one amoeba is alive.}\label{fig:collective}
\end{center}
\end{figure}
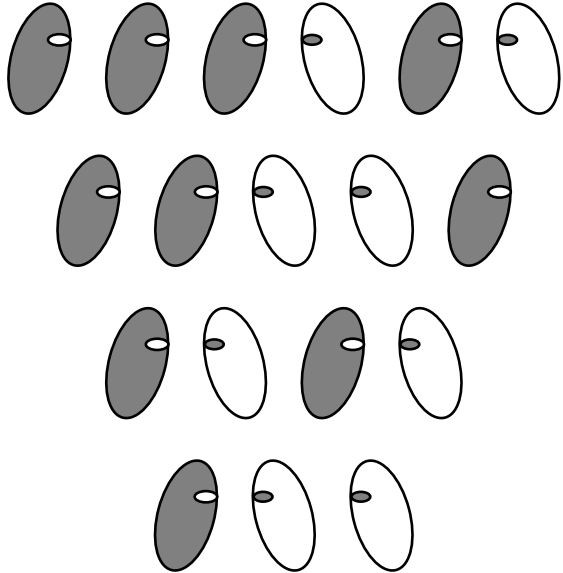

The black tribe moves by letting one of its members crawl rightwards across a number of white amoebae, while settling in the spot of a white amoeba, and thus pushing each bypassed amoeba one step to the left, whereas the white tribe moves by letting one of its members crawl leftwards across a bunch of black amoebae, while shifting the position of each bypassed amoeba one step to the right. Amoebae cannot bypass their own kind. When an amoeba reaches end of line, it cannot be played, and thus dies (of boredom), as indicated in the second line of Figure~\ref{fig:collective}. Unless, by moving, it bypassed all remaining amoebae, in which case its tribe will be rewarded eternal life. 
That is, a tribe that, at its turn, cannot move, because none of its members survived, loses this `evolutionary' combinatorial game.

\begin{figure}[ht]
\begin{center}
\begin{tikzpicture}[scale=0.15, line width=1pt]
    \path (-8,0) pic {amoebaL};
    \path (0,0) pic {amoebaL};
    \path (8,0) pic {amoebaR};
    \path (16,0) pic {amoebaL};
    \path (24,0) pic {amoebaR};
    \node[left] (A) at (34,0) {$\longrightarrow$};
\end{tikzpicture}
\begin{tikzpicture}[scale=0.15, line width=1pt]
    \path (0,0) pic {amoebaR};
    \path (8,0) pic {amoebaL};
    \path (16,0) pic {amoebaL};
    \path (24,0) pic {amoebaL};
    \path (32,0) pic {amoebaR};
    \draw[thick, red] (-5,5) -- (5,-5);
    \draw[thick, red] (-5,-5) -- (5,5);
\end{tikzpicture}\caption{The middle amoeba crawled to the left end. When an amoeba does not face any opponent, even at a far distance, it gets removed, because it cannot be used in the game by either player.}\label{fig:amorem}
\end{center}
\end{figure}
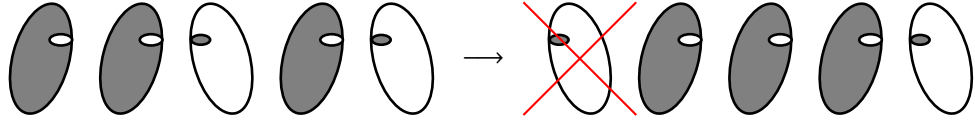

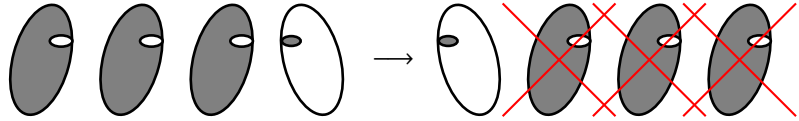
\begin{figure}[ht]
\begin{center}
\begin{tikzpicture}[scale=0.15, line width=1pt]
    \path (0,0) pic {amoebaL};
    \path (8,0) pic {amoebaL};
    \path (16,0) pic {amoebaL};
    \path (24,0) pic {amoebaR};
    \node[left] (A) at (34,0) {$\longrightarrow$};
\end{tikzpicture}
\begin{tikzpicture}[scale=0.15, line width=1pt]
    \path (0,0) pic {amoebaR};
    \path (8,0) pic {amoebaL};
    \path (16,0) pic {amoebaL};
    \path (24,0) pic {amoebaL};
    \draw[thick, red] (3,5) -- (13,-5);
    \draw[thick, red] (3,-5) -- (13,5);
    \draw[thick, red] (11,5) -- (21,-5);
    \draw[thick, red] (11,-5) -- (21,5);
    \draw[thick, red] (19,5) -- (29,-5);
    \draw[thick, red] (19,-5) -- (29,5);
\end{tikzpicture}\caption{By moving, the single white amoebae bypassed all remaining amoebae, and will be celebrated as a hero by its resurrected tribe.}\label{fig:eternal}
\end{center}
\end{figure}

\textsc{bipass} is a partizan combinatorial game played with several strips of stones.  The stones are either black ($\!\bl$) or  white ($\!\wh$). A move is to interchange a black and a white stone with the constraint that the black stone lies to the left of the  white, and that 
\begin{itemize}
\item Left cannot interchange them if there is a black stone between them; 
\item Right cannot interchange them if there is a white stone between them.
\end{itemize}
Equivalently, one can think of the black stones as Left's pieces, and by moving Left jumps a number of white stones immediately to the right of the black stone, by  shifting the selected white stones one step to the left, in order to fit in her jumped black stone. This description shows that {\sc bipass} is the game in the first paragraph. 

\begin{example}\label{ex:singleline}
Here is an example of single line play, where the black tribe starts, and wins in their second move: 
$$\bl\wh\wh\wh\bl\wh \; \stackrel{\bl}{\longrightarrow}\; (\!\wh\wh)\bl\wh\bl\wh \; \stackrel{\wh}{\longrightarrow}\;  \bl\wh\wh(\!\bl)\; \stackrel{\bl}{\longrightarrow}\; (\!\wh\wh)\,\widehat{\!\bl}$$
The brackets contain recently deceased amoebae, and $\widehat{\!\bl}$ indicates that the black tribe has been rewarded eternal life. 
\end{example}

We summarize the options of a position in the usual way for combinatorial games; thus the starting game in  Example~\ref{ex:singleline} has 6 options and is defined recursively by $$\bl\wh\wh\wh\bl\wh=\cg{\bl \wh\bl \wh,\bl \wh \wh\bl \wh,\bl \bl \wh,\bl\wh\wh\wh}{\bl\wh\wh\bl\wh,\bl\wh\wh\wh\wh\,}$$ 

Curiously enough Left, who plays black, wants many of Right's white pieces on the board, and an intuition is that Right's pieces correspond to Left's `game board'. Therefore she gains move advantages if we increase the number of black pieces (and vice versa). In the main part of the paper, we study the normal-play winning convention: i.e. a player who cannot move loses, and in this convention it is never bad to have more move options. The final section (Section~\ref{sec:mis}) will mention briefly the mis\`ere play convention, where a player who cannot move wins, but even within this class, it is not bad to have more options, provided that there is at least one.

The benefit of many move options in {\sc bipass}, is highlighted in the next example.
\begin{example}[Black Headed Larvae]
A \bip strip of one black stone, followed by seven single white stones to the right,  
$\bl\wh\wh\wh\wh\wh\wh\wh$, gives 7 options for Left, but only one option for Right. Moreover, Left has the empty game 0 as an option, but Right can only reduce the number of white stones one by one, and it would take him 7 consecutive moves to reach 0. Thus, in a disjunctive sum of games (similar to Figure~\ref{fig:collective}), Left can afford to wait 6 Right moves on this strip, because at this point, this strip will be $\bl\wh = \cgstar$, and only now she may want to react by eliminating the strip, for example, if this is the last remaining component in the original disjunctive sum.
\end{example}

These type of \bip strips have enough special properties that they do deserve a name. We call a strip of the form $\bl\wh\cdots\!\wh$ (with at least 1 white stone) by a {\em black headed larvae}. Similarly, a strip of the form $\bl\cdots\! \bl\wh$ (with at least 1 black stone) is a {\em white headed larvae}. The first picture in Figure~\ref{fig:eternal} is a white headed larvae. As, we will see, larvae have a couple of special properties that distinguishes them from other type of games. For example, they contain no {\em unit-bypass} for one of the players, a central concept to \bip (see Definition~\ref{def:unitbyp}).

\subsection{Context and concepts}
Regarded as a ruleset, \bip may be played in isolation, exclusively together with other \bip strips, as in the example in the first paragraph, or it may be played in the general context of normal-play games, i.e. mixed with other rulesets. In the first case, one can deduce various strategic observations by looking merely at the rules of \bipass, but in the more general context one must adhere to standard theory (i.e. game values/canonical forms) on normal-play games, as developed by Conway \cite{C}, Berlekamp, Conway and Guy \cite{WW}, and subsequently by Nowakowski, Wolfe and Albert \cite{LiP}, and most recently by Siegel \cite{S}. Although we seek to be self-sufficient, we invite readers new to the subject to study standard notation and terminology, as presented in those books. In particular, the atomic weight theory will come in handy, and we will review the basics from this theory that applies to this paper in Section~\ref{sec:AW}. 

In \bipass, if Left has a move, there is a black stone that can be swapped with a white stone. The stone neighboring this black stone, to the right, is a white stone, and hence Right also has a move. Similarly, if Right has a move then so does Left. Therefore, merely parity considerations determine the winner, and we saw this already in the example in the first paragraph. Therefore, \bip is an {\em all-small} game \cite[p. 229]{WW}, \cite[p. 101]{LiP}, \cite[p.83]{S}, where such games are called dicots]. All-small games are in sharp contrast with rulesets such as {\sc amazons} \cite{Amaz}, {\sc domineering} \cite{YB} and {\sc go} \cite{Go}, where a player wins by ``gaining territory'' in which the other player cannot move. In \bipass, neither player can gain such an advantage; in particular, the game values are all infinitesimals \cite[p. 36-37]{WW}, \cite[p.100]{LiP}, \cite[p.83]{S}.

The game values of \bip become complicated even for small positions (See final row in Table~\ref{tab:values} and Example~\ref{ex:compl}.). Atomic weight, abbreviated \aw, \cite[Ch. 8]{WW},\cite[p.197]{LiP}, \cite[p.151]{S}, is an efficient approximation for infinitesimals arising from all-small games. We use these to solve many multi-strip games, and moreover, atomic weight is an essential tool when playing a disjunctive sum of \bip and any other all-small games, such as {\sc clobber} \cite{AGNW},  {\sc cutthroat stars} \cite{LiP,MuNo},
{\sc hackenbush sprigs} \cite{MMN}, {\sc partizan
euclid} \cite{McKayN2013} and {\sc yellow-brown hackenbush} \cite{Berl}. 

The fundamental concepts of a normal play combinatorial game is the {\em outcome function}, the {\em disjunctive sum} of games,  and the {\em game value}.  The outcome of a game $G$, $o(G)$, is \L (\R) if Left (Right) wins independently of who starts, and it is N (P) if the curreNt (Previous) player wins. If the outcome is $X$, then we often say that the game is an $X$-position. The convention is that `Left wins' $>$ `Right wins'. Therefore the outcomes are partially ordered with $\L>\P>\R$ and $\L>\N>\R$ but $\P\cgfuzzy \N$. Let $G$ and $H$ be combinatorial games. The disjunctive sum operator `$+$' is defined, recursively, by the current player's options in the disjunctive sum $G+H$. If Left starts, then an option is of the form $G^L+H$ or $G+H^L$, where $G^L$ denotes a typical Left option in the game $G$, and similar for Right. This definition defines a partial order of normal play games: $G\ge H$ if $o(G+X)\ge o(H+X)$ for any normal play game $X$. A fundamental theorem of normal play games gives that $G\ge H$ if and only if Left wins the game $G+(-H)$ if Right starts, where $-H$ is the game where the players have (hereditarily) swapped positions in $H$. This implies that $G\equiv H$ if and only if $G-H$ is a P-position.\footnote{In this paper we designate the symbol `$\equiv$' to specify equivalence of games, whereas `$=$' has multiple uses.} Another fundamental theorem is that there is a unique game of smallest birthday (rank) in each equivalence class, which is referred to as the {\em game value}, of this class. 

Let us summarize the contribution of this paper.
\begin{itemize}
\item Section~\ref{sec:singlestrip} answers the question: ``Who wins a single \bip strip?'';
\item In Section~\ref{sec:canforms}, we study canonical forms of some parametrized \bip strips;
\item In Section~\ref{sec:main}, we demonstrate how a simple surjective rule assigns integer atomic weights to \bipass;
\item Section~\ref{sec:AW0} bounds harder instances of \bipass;
\item In Section~\ref{sec:imposs}, we show that $*2$ does not appear as a value in \bipass;
\item In Section~\ref{sec:mis}, we analyze mis\`ere \bipass.
\end{itemize}

\section{Atomic weight and $\Delta$-excess}
Much of the usefulness of atomic weights (\aw) is given in Theorem \ref{thm:AWprops}. Our main result is Theorem~\ref{thm:bipAW} which relates \bip and atomic weights. We will detail the definitions etc. in  Section~\ref{sec:AW}. In this section, we will begin to discuss how to use \aw\ in terms of \bipass. 
\begin{theorem}[Atomic Weight Properties, \cite{S}]\label{thm:AWprops} 
Let $g$ and $h$ be all-small games.  Then
\begin{enumerate}[(i)]
\item $\aw(g+h) = \aw(g) + \aw(h)$;
\item $\aw(-g)=-\aw(g)$;
\item If $\aw(g) \geqslant 1$, then $g\cggfuz 0$ (Left wins playing first);
\item  if $\aw(g)\geqslant 2$, then $g > 0$ (Left wins).
\end{enumerate}
\end{theorem}

In particular (iv) is the raison d'\^ etre for atomic weight, and it is popularly called ``the two-ahead-rule''. We will have plenty use for it.

Let $b(s)$ and $w(s)$, respectively, denote the number of black and white stones on a given strip $s$. 
Since we assume that all stones are alive, we have that $b(s) > 0$ if and only if $w(s) > 0$. Let $|s|$ be
the number of stones on $s$.

\begin{definition}[$\Delta$-excess]\label{def:Delta} 
Consider a \bip strip, where all pieces are alive. Then $\Delta(s) = w(s) - b(s)$ denotes the
excess of white stones on $s$. Let $g = g_1 + g_2 + \cdots + g_n$ be a sum of \bipass strips. Then $\Delta(g) = \sum \Delta(g_i)$. When the underlying game is understood, let $\Delta = \Delta(g)$.
\end{definition}

Table \ref{tab:values} gives the outcomes, values, $\Delta$-excesses and atomic weights for strips with up to 5 stones and $b(s)\leq w(s)$.

\begin{table}[htb]
\caption{Outcomes, values, excesses and atomic weights of positions up to 5 stones. As usual $\star=\cg{0}{0}$, $\up=\cg{0}{\star}$, $\cgdoubleup = \up +\up$, $\cgdoubleup\star = \cgdoubleup +\star$, and so forth.}
\begin{center}
\begin{tabular}{|c|c|l|l|c|}
Position&Outcome&Value&$\Delta$&\aw\\
\hline
$\bl\wh$&N&$\cgstar$&0&0\\
$\bl\wh\wh$&L&$\cgup$&1&1\\
$\bl\wh\wh\wh$&L&$\cgdoubleup\cgstar$&2&2\\
$\bl\wh\bl\wh$&N&$\cgstar$&0&0\\
$\bl\bl\wh\wh$&N&$\{\cgstar,\cgup\mid \cgstar,\cgdown\}$&0&0\\
$\bl\wh\wh\wh\wh$&L&$\cgtripleup$&3&3\\
$\bl\wh\wh\bl\wh$&L&$\cgup$&1&1\\
$\bl\wh\bl\wh\wh$&L&$\{\cgdoubleup\cgstar\mid\cgup,\{\cgstar,\cgup\mid \cgstar,\cgdown\}\}$&1&1\\
$\bl\bl\wh\wh\wh$&L&$\{0|\{\cgstar,\cgup\mid \cgstar,\cgdown\},\{\cgdoubleup\cgstar\mid\cgup,\{\cgstar,\cgup\mid \cgstar,\cgdown\}\}\}$&1&1\\
\hline
\end{tabular}
\end{center}
\label{tab:values}
\end{table}%

Note in the table that the $\Delta$-excesses coincide with the atomic weights. Our main theorem asserts that this readily generalizes.
 
\begin{theorem}\label{thm:bipAW}
Let $g$ be a disjunctive sum of  {\sc bipass} strips. Then $\aw(g)=\Delta(g)$.
\end{theorem}

From this result we gather for example that the composite position in Figure~\ref{fig:collective} has atomic weight -1, and hence, already by the general atomic weight theory, Right wins playing first. But, in fact, by the restriction to \bip (Proposition~\ref{prop:singlestrip}, iv), we will be able to  conclude that Right wins independently of who starts.

A very basic lemma follows directly by combining these two results. Let us first introduce a main concept of \bipass.

\begin{definition}[Unit-bypass]\label{def:unitbyp}
Let $s$ be a \bip strip with at least two black stones. Then, if the right most black stone is moved to the right end, the move is called a {\em \ubp}. The analogous terminology holds for white stones.  A {\em neighbor-bypass} is a move that swaps a stone with its neighbor.
\end{definition}

Figure~\ref{fig:amorem} depicts a \ubp. The requirement that there are at least two black stones is essential, because the terminology has been introduced for a strategic reason. Namely, the move $\bl\wh\rightarrow 0$ may seem to be of the same form, but is excluded since it does not affect the $\Delta$-excess. Every \ub `improves' the $\Delta$-excess for the current player. (Only rarely a $\Delta$-increase is bad for Left.) The type of move in the definition assures that there remains a non-empty alive strip. 

We saw that a \ub is a correct winning move in many instances. But, a neigbor-bypass can be a winning move, in cases when there is no \ubp. This happens for example in the N-position $\bl\wh\wh +\bl\wh$, where both players, as starting players, can win by making a neigbor-bypass. Coincidently, Right's winning \nb is also a \ubp. See also Example~\ref{ex:nonendbypass}, where a \nb is the unique winning move in spite there existing a \ubp. But Left loses, if she plays the only available \ubp. 

\section{More on the $\Delta$-excess, and a single strip solution.}\label{sec:singlestrip}
Play on a single strip is very simple, and relies only on observing changes in $\Delta$. We get repeated use of the following basic observation, which formalizes our use of a \ubp.

Recall that a strip with only one black piece is a larvae. A black headed larvae of length $k\geq 3$ has atomic weight $k-2$ which is the greatest atomic weight possible for a strip with $k$ pieces. Right's unique move is a \ubp, which decreases the atomic weight by 1, whereas Left can move to a position with atomic weight $j$ with $0\leq j<k$.

\begin{lemma}\label{lem:oneunit}
Consider a non-empty {\sc bipass} strip $s$. If it contains at least 2 black stones, there is a unique Left move that increases $\Delta$. Moreover, this move increases $\Delta$ by precisely one unit. Otherwise, if $s$ contains exactly one black stone and $|s|>2$, then $\Delta$ decreases by moving. \end{lemma}
\begin{proof}
An increase of $\Delta$ means that some black stones have been eliminated. Since by moving, Left moves exactly one black stone, then $\Delta$ can increase by at most one unit. Since $s$ is non-empty, then there is at least one black stone to the left of a white stone. Consider the rightmost such stone. It has at least one white stone to the right, since otherwise it would have been  eliminated. Hence Left can remove this black stone and shift the rightmost white stones one step to the left.  This type of moves increases $\Delta$ by one unit, unless there is exactly one black stone in $s$, in which case $\Delta$ weakly decreases to 0 (since in this case, $\Delta(s)\ge 0$).

Next, assume that Left moves any other black stone. Then, this black stone cannot be moved to the rightmost slot, since that means it would have jumped the rightmost black stone. Hence it cannot decrease $\Delta$. 

The last statement is obvious, since the strip must be of the form $\bl\wh\cdots\! \wh$, with at least 2 white stones, and any move by either player will decrease $\Delta$ by at least one unit. 
\end{proof}

The type of move indicated by the proof of Lemma~\ref{lem:oneunit}, with an increase of $\Delta$, is of course a \ubp: the strip has at least two black stones, and Left moves the right-most black stone to the end of the strip. This type of move is central to the analysis of \bipass. 

The solution of \bip on a single strip is very simple.

\begin{proposition}\label{prop:singlestrip}
Let $s$ denote a single \bip strip. 
\begin{enumerate}[(i)]
\item The first player loses if $|s|=0$ (P-position); 
\item The first player wins if $\Delta(s)=0$ and $|s|>0$ (N-position).
\item Left wins if $\Delta(s)>0$ (L-position); 
\item Right wins if $\Delta(s)<0$ (R-position); 
\end{enumerate}
\end{proposition}
\begin{proof}
By symmetry, it suffices to prove (i) - (iii). The first item is obvious, so consider (ii) and (iii).\\

\noindent {\em Case 1, Left does not have a unit-bypass as an option:} In this case the strip is a black headed larvae, i.e. of the form $\bl\wh\cdots \wh$. Therefore, if $\Delta(s) = 0$, $s=\bl\wh$, an N-position. Otherwise, Left wins playing first or second, by eliminating the strip.\\

\noindent {\em Case 2, Left has a unit-bypass as an option:} 
We claim that, if $\Delta(s)\geq 0$ then Left wins, going first, by playing the \ubp. Namely, this increases 
$\Delta(s)$ by one unit, and then use (iii), by induction. If $\Delta(s) > 0$, then, by (ii) and induction, Left wins going second. Namely, by Lemma~\ref{lem:oneunit}, Right can, in his first move, at most decrease $\Delta$ by one unit. 
\end{proof}

Hence, we get the following result.
\begin{corollary}\label{cor:P0}
The only P-position of a single \bip strip occurs for the empty game. 
\end{corollary}
\begin{proof}
This is immediate by Proposition~\ref{prop:singlestrip}.
\end{proof}

This behavior is a bit unusual but does occur in other games, `single stalks' of  \textsc{hackenbush}, `one line' of \textsc{toppling dominoes}, `one heap' of \textsc{wythoff partizan subtraction} \cite{La} and for the case of `one star' in \textsc{cutthroat stars}. 

As a guideline for many situations, a \ub is a good move. But there are severe exceptions to this na\"ive intuition.
\begin{example}\label{ex:nonendbypass}
Figure~\ref{fig:collective2} shows a composite position, where it is non-optimal to play a \ubp. We have that $\Delta = 0$, and although a \ubp, by say player Left, on the top strip gives $\aw= 1$ (the value becomes $\up\star\cggfuz 0$) Right can counter by playing to $\bl\wh+\bl\wh = \star+\star$. Left's unique winning move from $\bl\bl\wh\wh+\bl\wh$ is to $\bl\wh\bl\wh+\bl\wh\equiv \star+\star\equiv 0$. (See Table~\ref{tab:values}.)
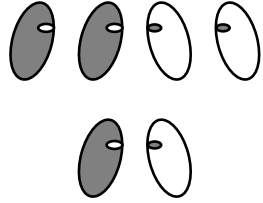
\begin{figure}[ht]
\begin{center}

\begin{tikzpicture}[scale=0.7, every node/.style={transform shape}, line width=1pt]
    \path (-1.3,0) pic {amoebaL};
    \path (0,0) pic {amoebaL};
    \path (1.3,0) pic {amoebaR};
    \path (2.6,0) pic {amoebaR};
\end{tikzpicture}\vspace{4 mm}

\begin{tikzpicture}[scale=0.7, every node/.style={transform shape}, line width=1pt]
    \path (2.6,0) pic {amoebaL};
    \path (3.9,0) pic {amoebaR};
\end{tikzpicture}\caption{A collective for which the unique winning move is a neighbor-bypass.}\label{fig:collective2}
\end{center}
\end{figure}
Of course, if the upper strip in Figure~\ref{fig:collective2} is played alone, then a \ub is the unique winning move if Left starts.
\end{example}
Hence, a na\"ive strategy that always plays a \ub can fail. And any `na\"ive strategy' is bound to fail in general.
\begin{example}\label{ex:compl}
This example points towards arbitrarily complex strategies, and we have used CGSuit \cite{CGSuit}. Consider the game  $\bl\bl\bl\wh\wh\wh$.
This (symmetric) game is the unique game with exactly three options for each player and where, in fact, all options survive in canonical form.  For readability, let us call the game $g=\pm(*,\up) = \cg{*,\up}{*,\cgdown}$ (a game that occurs frequently in {\sc bipass}). Here are the Left options:
\begin{itemize}
\item $g$, 
\item $\{0 | g,\{\cgdoubleup *|\up, g\}\},$ 
\item $\pm(g,\{0|g,\{\cgdoubleup*|\up,g\}\}).$
\end{itemize}
It is easy to justify why all options survive in canonical form, namely one can see that both $\bl\bl\wh\wh\wh+\bl\bl\wh\bl\wh\wh$ and  $\bl\bl\wh\wh\wh+\bl\wh\bl\bl\wh\wh$ are N-positions (find mirroring moves!). Hence, the \ub is a losing Right move from the position $\bl\bl\bl\wh\wh\wh+\bl\bl\wh\bl\wh\wh$. Instead he should play to the mirror game. 

Here is a game with exactly 4 options for each player, where all Left options survive in the canonical form representation: $\bl\bl\bl\wh\wh\wh\bl\wh$ (but not all Right options). Can one find \bip strips that justify the survival of the canonical form options, or do we have to look further in the general class of all-small games? However, in the game $\bl\bl\bl\bl\wh\wh\wh\wh$, only two options survive, for each player, in the canonical form representation. That is, in any  play situation (all-small games) the other two options can be ignored.  Are there arbitrarily large \bipass\ games such that all (Left) options survive in the canonical form?
\end{example}

\section{Canonical forms of some \bip strips}\label{sec:canforms}
In this section, we characterize an infinite class of \bip game values (canonical forms). 
Standard CGT-notation is `up' $\cgup=\cg{0}{\star}$ (and `down' $\cgdown=\cg{\star}{0}$), so that $\cgup\star=\{\star,0\mid 0\}$ (and $\cgdown\star=\{0\mid \star,0\}$), either being a standard representations of `a unit' in atomic weight theory (Section~\ref{sec:AW}). The latter will often be the convenient representation in the proofs to come. Therefore, to simplify notation, we write $\upstar=\cgup\star$, and $k$ such games in a disjunctive sum is conveniently $k\upstar$, and moreover, we will abbreviate $k\upstar+\star=k\upstar\star$. (With this notation, $\upstar\star=\cgup$, $\cgup\star=\upstar$, $\downstar\star=\cgdown$ and $\cgdown\star=\downstar$.)

The next result is not new but it is worth re-iterating it with the new notation.
\begin{lemma}[\cite{LiP}]\label{lem:LiP}
For any integer $k>0$, $k\upstar\star\equiv \{0\mid(k-1)\upstar\star\}$.
\end{lemma}
\begin{proof}
Induction: $k=1$ gives $\up=\{0\mid\star\}$, which is the definition of $\up$. Suppose the statement holds for $k\ge 1$, and we must prove that $(k+1)\upstar\star=\{0\mid k\upstar\star\}$.

By induction, we have that $k\upstar\star +\upstar=\{0\mid(k-1)\upstar\star\}+\upstar$, and so, it suffices to prove that the first player loses $ \{0\mid k \upstar\star\} + \{(k-1)\downstar\star\mid 0\}+\downstar$.  

Suppose Left starts, and plays to $\{(k-1)\downstar\star\mid 0\}+\downstar$. Then Right can respond to $\{(k-1)\downstar\star\mid 0\}$, which is bad for Left. 
If Left plays to $ \{0\mid k\upstar\star\} + \{(k-1)\downstar\star\mid 0\}$, then Right responds to $ k\upstar\star + \{(k-1)\downstar\star\mid 0\}$ and whatever Left does, Right can play to $((k-1)\upstar\star) + (k-1)\downstar\star =0$, which is losing for the first player, which is Left. 
The arguments for Right starting lead to similar situations.
\end{proof}

Let $\bl^n$ and $\wh^n$ denote $n$ consecutive black and white stones in a \bip strip, respectively.

\begin{theorem}\label{thm:bipass}
Consider a single {\sc bipass} strip $s$. 
Then, for any nonnegative integers $n$ and $k$, 
\[s=\bl\wh^{n+k}\bl^n\wh \mbox{ if and only if }s\equiv k\upstar \star\]
 and similarly
\[s=\bl\wh^{n}\bl^{n+k}\wh \mbox{ if and only if } s\equiv k\downstar \star.\]
\end{theorem}

\begin{proof}
The second statement is obtained by taking the negative of the first game, and so it suffices to prove the first. In this result, a position is succinctly described by an ordered pair of nonnegative integers. Whenever convenient, we abbreviate $(n,k)=\bl\wh^{n+k}\bl^{n}\wh$. In particular, it can be easier to see the explicit induction this way. 

First we show that if  $s=(n,k)$ then $s\equiv k\upstar\star$.
We induct on $2n+k=|s|-2$. 

When $2n+k = 0$, then $\bl\wh=\star$. Assume that the claim is true for all strips $s=\bl\wh^{n+k}\bl^{n}\wh$, with $2n+k < m$, for some $m > 0$, and set $0<2n+k=m$, i.e. $n>0$ or $k>0$.
 By Lemma~\ref{lem:LiP}, we have that $k\downstar\star\equiv \cg{(k-1)\downstar \star}{0}$. It suffices to prove that $(n,k)+k\downstar\star \equiv 0$, i.e. that the first player loses $(n,k)+\cg{(k-1)\downstar \star}{0}$.\\

\noindent Case 1: Left starts.

\begin{enumerate}
\item Left plays to $(n,k) +\ (k-1)\downstar\star$. Right responds with a \ub to 
 $(n,k-1)+(k-1)\downstar\star\equiv 0$, by induction.
\item Left plays an \ub to $(n,k-1)+k\downstar\star$. Right responds with a \ubp, resulting in $(n-1,k)+k\downstar\star\equiv 0$, by induction.
\item Left plays a non-\ub to either $H=(n,\ell)+k\downstar\star$,  $\ell<k$, or 
 $H'=(p,p+q)+k\downstar\star$, $p<n$, $p+q=n$. In the first case, by induction, $H\equiv\ell\upstar\star +k\downstar\star\equiv(k-\ell)\downstar<\cgfuzzy0$. 
In the second, by induction, $H'\equiv(q+k)\downstar<\cgfuzzy0$, since $q+k>0$. In, either case, since Right starts, Right wins.
\end{enumerate}

\noindent Case 2: Right starts.
\begin{enumerate}
\item If Right plays to $(n,k)+0$, then Left wins by Proposition~\ref{prop:singlestrip}, since $n+k>0$.
\item Suppose Right plays the \ub to $(n,k-1)+k\downstar\star$. By induction 
this is equivalent to $(k-1)\downstar\star+k\downstar\star=\downstar$, and Left wins going next. 
\item Suppose Right plays to $\bl\wh^{n+k}\bl^{n-j}\wh+k\downstar\star$, $j>0$. By induction,
this is equivalent to $(k+j)\upstar\star + k\downstar\star =j\upstar\star$
which Left wins going first.
\end{enumerate}

Now suppose that a {\sc bipass} strip $s\equiv k\upstar\star$, and we must show that $s=\bl\wh^{n+k}\bl^n\wh$.\footnote{This is the literal form of $s$ of course.} We first show that $\Delta(s)=k$.

We demonstrate that, if $\Delta(s)=j >k$, then $s+k\downstar\star\not\equiv 0$. Namely, we show that Left wins by playing first to $s+(k-1)\downstar\star$. By Proposition~\ref{prop:singlestrip}, Right loses if he moves to $s+0$ (see Lemma~\ref{lem:LiP}) so he must play an \ub to say $s^R+(k-1)\downstar\star$, and note that $\Delta(s^R) = j-1$. Similarly, if, at each move, Left decreases the number of $\downstar$s, then Right must play an \ub until, after Right's move the position is $s'+0$. Since $\Delta(s')=j-k>0$, by Proposition \ref{prop:singlestrip}, Left wins $s+k\downstar\star$, going first. 

Similarly, if $\Delta(s)=j <k$  then Right can win $s+k\downstar\star$, going first, by making an \ub. 
Regardless whether Left makes an \ub, or eliminates a $\downstar$, Right continues making \ub es until $s$ has been reduced to $0$. After Left's response, the position is $\ell\downstar\star$, 
 for some $\ell\geq k-j$. Since $\ell\downstar\star$ is negative, Right wins going first. 

Therefore we conclude that $\Delta(s) = k$. But apart from that information, $s$ might be arbitrary. 
Again, consider playing $s+k\downstar\star\equiv 0$.

Suppose that Left goes first. We will show that Left can force a win, unless $s$ is of the specified form. 

If $|s|=2$, then  $s=\bl\wh=(0,0)\equiv \star$, the previous player wins and $s$ is of the specified form, with $n=k=0$ (the cases $\bl\bl$ and $\wh\wh$ are not alive). If $|s|=3$, then $s=\bl\wh\wh=(0,1)$. If $|s|=4$, there are 3 cases, $s=\bl\bl\wh\wh$, $s=\bl\wh\bl\wh=(1,0)$ or $s=\bl\wh\wh\wh=(0,2)$. In the first case Left wins $s+*$ as indicated in Example~\ref{ex:nonendbypass}, using the non\ub move, and the other two cases, $(1,0) +\star$ and $(0,2)+2\downstar\star$ are P-positions. 

In the last part of this proof, we will required a classical result of atomic weight theory, namely the 2-ahead rule, together with the main result.

Let us proceed by induction. Suppose that $s$ is not of the specified form. Claim: An \ub by Left, provokes an \ub by Right. This follows, by combining Theorems~\ref{thm:AWprops} and \ref{thm:bipAW}. Namely, because $\Delta=0$ to start off with it would turn $\Delta =2$ if Left plays another \ub after Right neglecting to do so.

If such a pair of moves results in a position not of the specified form, then, we are done by induction. But, we claim that after Left's \ubp, Right's provoked \ub cannot produce a position of the specified form, unless  $s$ is as in Example~\ref{ex:nonendbypass}. By way of contradiction, assume that the position, after Right's response to Left's opening move, is of the form $s=(n,k)$. By assumption both players played a \ubp. That is each player eliminated  one of their own pieces: say $s\rightarrow s'\rightarrow \bl\wh^{n+k}\bl^n\wh$. Hence $s'=\bl\wh^{n+k+1}\bl^n\wh$, and $s=\bl\wh^{n+k+1}\bl^{n+1}\wh$, which contradicts that $s$ be not of the specified form. 
\end{proof}

\section{Atomic weight theory}\label{sec:AW}

The game of \textsc{bipass} has so far been analyzed to the level of playing one strip or the disjunctive sum of  very specific strips.  To play \textsc{bipass} in disjunctive sum with other all-small games it is convenient to consider an approximation to the value. Here, we review well-known theory on atomic weights ($\aw$). 

In cases where it is easy to compute the atomic weights, but the canonical form is unintelligible, this theory guides us in answering the general question of ``which game to prefer'', while just knowing the outcome almost never answers this question.

The value of any strip $s$, as in Theorem \ref{thm:bipass}, is $\Delta(s)\cdot\upstar\star$, and we will see that $\aw(s)=\Delta(s)$, which continues to hold if we replace $s$ with a sum of {\sc bipass} strips. This is the essence of atomic weights in {\sc bipass}, and this will be proved in Section~\ref{sec:main} (Theorem~\ref{thm:bipAW}). 

Let us recall the atomic weight of an all-small combinatorial game.

\begin{definition}[Far Star] 
The {\em far star}, $\cgfarstar$, is an arbitrarily large Nim-heap, i.e. both players can move to any {\sc nim}-heap from $\cgfarstar$, with the additional property $\cgfarstar +\cgfarstar = \cgfarstar$. 
\end{definition}

Equivalence modulo $\cgfarstar$ is obtained as follows. 
\begin{definition}[Equivalence Modulo $\cgfarstar$]
Let $G,H$ be normal play games. Then $G =_\cgfarstar H$, if, for all games $X$, $o(G+X+\cgfarstar)=o(H+X+\cgfarstar)$. 
\end{definition}

\begin{theorem}[Constructive $\cgfarstar$-equivalence]\label{Confs}
Let $G,H$ be normal play games. Then $G =_\cgfarstar H$ if and only if $\downstar <G-H<\upstar$.
\end{theorem}

The atomic weight is well defined for \emph{all-small} games. In mis\`ere play (and other classes of games) all-small is instead called {\em dicot}. The normal-play naming is due to Conway and helps intuition to remind us that each all-small game is an infinitesimal (this does not hold in mis\`ere play). 

Let $X$ be a set. Then $X+y=\{x+y:x\in X\}.$ If $X$ is a set of all small games, let $\aw(X)=\{\aw(x): x\in X\}$.

\begin{example}
Let $n$ be an integer. By Theorem~\ref{Confs}, $\cgstar n$ has atomic weight 0, since $\downstar < \cgstar n-0<\upstar$.
\end{example}

In this example we had to guess an atomic weight and then verify. Next we will show how to recursively compute the atomic weight of a game.

The product of a game $G$ and $\cgup$ is: $0\cgnmultiply \cgup=0$; $n\cgnmultiply \cgup=\cgup + (n-1)\cgnmultiply \cgup$, in case $G=n$ is an integer. Otherwise $G\cgnmultiply \cgup = \{\GL+\cgdoubleup\cgstar \mid \GR+\cgdoubledown\cgstar \}$. 

\begin{lemma}[\cite{S}]
Consider any normal play game $G$.
\begin{itemize}
\item If $G\cgnmultiply \up\ge_\cgfarstar 0$ then $G\ge 0$. 
\item If $g$ is all-small then there is a unique $G$ such that $g=_\cgfarstar G\cgnmultiply \up$.
\end{itemize}
\end{lemma}
The first item implies the second, with some work, and we use the uniqueness to define atomic weight.
\begin{definition}[Atomic Weight]\label{AW} 
The atomic weight of an all-small game $g$ is the unique game $G = \aw(g)$ such that $G\cgnmultiply \cgup =_\cgfarstar g$.
\end{definition}

There is a recursive formula for computing atomic weights. If $X$ is a set of all-small games, let $\aw(X)=\{\aw(x): x\in X\}$.

\begin{theorem}[Constructive Atomic Weight, \cite{S}]\label{thm:AWcon} 
Let $g$ be an all-small game, and let
$$G = \cg{\aw(\gL)-2}{\aw(\gR)+2}.$$ Then $\aw(g)=G$, unless $G$ is an integer. In this case, compare $g$ with far star. If
\begin{itemize}
\item $g\cgfuzzy \cgfarstar$, then $\aw(g)=0$;  
\item $g < \cgfarstar$, then $\aw(g)=\min\{n\in \mathbb Z: n\cggfuz G^L\}$; 
\item $g > \cgfarstar$, then $\aw(g)=\max\{n\in \mathbb Z: n\cglfuz G^R\}$.
\end{itemize}
\end{theorem}

\begin{example}
 We compute $\aw(\cgstar n)$ explicitly. By induction, we get $G=\{-2\mid 2\}=0$, which is an integer. Since $\cgfarstar \cgfuzzy \cgstar n$ (because $\cgfarstar +\cgstar n=\cgfarstar\cgfuzzy 0$), we get $\aw(\cgstar n)=0$.
\end{example}

 We are now ready to apply atomic weight theory to \bipass.

\section{A main {\sc bipass} theorem}\label{sec:main}
Our main result gives a simple formula for computing the atomic weight of {\sc bipass}. We restate Theorem~\ref{thm:bipAW}. 

\begin{theorem2}
Let $g$ be a disjunctive sum of  {\sc bipass} strips. Then $\aw(g)=\Delta(g)$.
\end{theorem2}
\begin{proof} 
It suffices to prove the result for one strip $s$, since the case for several strips then follows by additivity of $\Delta$ and atomic weights (see Theorem \ref{thm:AWprops}) respectively. Recall that $\Delta(s)=w(s)-b(s)$, the excess of white stones on the strip $s$. The base cases are the empty strip, together with the case $|s|=2$, when $s=\cgstar$, and in either case $\Delta(s)=\aw(s)=0$. We separate the proof in two cases, $\Delta>0$ and $\Delta=0$ (and then $\Delta<0$ will follow by symmetry).\\

\noindent Case 1, $\Delta>0$: Suppose that the result holds for all options of $s$, and denote $\Delta(s) = n > 0$, and we must prove that $\aw(s)=n$. We consider two cases, either (i) $b(s)=1$, or (ii) $b(s)>1$. 

For (i) note that the single black stone must be the leftmost stone, and moreover there are at least two white stones to the right of the single black stone. Therefore, for each move (Left or Right), $\Delta$ will decrease by at least one unit, until it reaches 0. 
Consider $s$: Right has only one option, and hence, with $g=s$ as in Theorem~\ref{thm:AWcon}, 
\begin{align}\label{eq:Rightmove}
G^R=n-1+2=n+1. 
\end{align}
Similarly, by induction combined with domination of Left options,
\begin{align}\label{eq:Leftmove}
G^L=n-1-2=n-3. 
\end{align}

Thus $G=n-2$, an integer. By combining \eqref{eq:Rightmove} with the last part of Theorem~\ref{thm:AWcon}, it suffices to prove  that $g>\cgfarstar$. But Left wins $g+\cgfarstar$ playing first, for example, by moving to $g$, since Left wins from $g^R$, by the choice of game. If Right starts and plays to $g^R+\cgfarstar$, note that $\Delta(g^R)\ge 0$. If $\Delta(g^R)=0$, then, since Right moved, the strip is $g^R=\bl\wh$, and so Left wins by playing to $g^R+\cgstar=\cgstar+\cgstar$, and otherwise Left wins by the previous argument. 

For (ii), note that since there are at least two  black stones, there are three types of options, either $\Delta$ stays the same, $\Delta$ increases by precisely one (Left plays \ub), or perhaps $\Delta$ decreases by precisely one (Right plays \ubp). If Right starts, and does not play \ubp, then by induction, Left can use this by playing an \ub and achieve $\aw\ge 2$. If Left starts, she can play a \ub and achieve at least $\aw\ge 1$ in her next move, which suffices to win, by Theorem~\ref{thm:AWprops}.

Thus, since each player would play an \ub, we get 
\begin{align}
G &= \cg{n+1-2}{n-1+2}\\ 
&= \cg{n-1}{n+1}=n,
\end{align}
which is again an integer. Thus we compare $g$ with $\cgfarstar$, and prove that $g>\cgfarstar$, which is analogous to the final paragraph of case (i).\\ 

\noindent Case 2, $\Delta = 0$: For the case $\Delta(s)=0$ with $|s|>2$, Left can play to $\Delta(s^L)=\aw(s^L)=1$, and Right can play to $\Delta(s^R)=\aw(s^R) = -1$ (by induction). If these options are optimal, we get $G=\cg{-1}{1}=0$, which is an integer. Similarly, $G$ remains an integer even if Left and/or Right does not play a \ubp. Since, we want to show that $\aw(s) = 0$, by Theorem~\ref{thm:AWcon}, it therefore suffices to show that $g\cgfuzzy \cgfarstar$, that is that the next player wins $g + \cgfarstar$. 

Suppose that Left starts by playing such that $\Delta(s^L)=\aw(s^L)=1$. Then, Right must respond to decrease the atomic weight, for otherwise Left wins, by using induction and Theorem~\ref{thm:AWprops}. At each sub-position with $\Delta(s)$ even, Left plays a \ubp, and Right most respond with a \ubp. At some point, Right will play to $\bl\wh+\cgfarstar$, and then Left responds with $\bl\wh+\cgstar\equiv 0$. The argument is analogous if Right starts.
\end{proof}

By this result, if there are several strips, Proposition \ref{prop:singlestrip} can easily be extended to the cases where the $\Delta$-excess is at least two. Moreover, the winning strategy is profoundly simple. Play any \ub on any strip, and this will maintain the \aw-lead. 
The next result rephrases the atomic weight properties in terms of $\Delta$-excess.
\begin{corollary}\label{cor:disjsumstrips}
Consider a disjunctive sum of strips $g$.
\begin{itemize}
\item If $\Delta(g)\geq 2$ then $g>0$;
\item If $\Delta(g)=1$ then $g\cggfuz 0$.
\item If $\Delta(g)=-1$ then $g\cglfuz 0$.
\item If $\Delta(g)\leq- 2$ then $g<0$.
\end{itemize}
Moreover, in case of $\Delta(g)\geq 1$ then a Left-winning strategy is to play a \ub on any strip, where this is still possible. If there is no \ubp, then Left may play a \nb on any strip until the game ends. 
\end{corollary}
Note that a disjunctive sum $g$, with $\Delta(g)=1$, may be a win for Right, playing first. This happens for example in the game $g=\bl\wh\wh+\bl\wh$. Moreover, Corollary~\ref{cor:disjsumstrips} does not cover the case when $ \Delta=0$. This is the subject of the next section.

\section{\bip with atomic weight 0}\label{sec:AW0}
Games with no \ub are very special. 
\begin{lemma}\label{lem:AW01}
Let $g$ be a disjunctive sum of \bip strips. If $\Delta(g)=0$ and Left has no \ubp, each strip is $\bl\wh$. If $\Delta(g)=1$ and Left has no \ubp, each strip is $\bl\wh$, except for exactly one of the strips, which is $\bl\wh\wh$. In general, if Left has no \ubp, then each strip is a black headed larvae. 
\end{lemma}

\begin{proof}
 If Left does not have a \ub in any strip in $g$, then each strip  
is of the form $\bl\wh\cdots\wh$, a black headed larvae. If $\Delta(G)=0$ then the number of black stones equals the number of white stones. It follows therefore that each strip is of the form $\bl\wh$. 

Similarly, if $\Delta(G)=1$, and Left has no \ubp, then there is exactly one strip with two white and one black stone. 
\end{proof}

It follows that, if no player has a \ubp, then $\Delta=0$, and so each strip is of the form $\bl\wh$.

In case of \bip, we can sometimes strengthen the atomic weight properties to include games with atomic weight 0. Namely, if \bip is played on an odd number of strips, with total atomic weight 0, then the first player wins. 

\begin{lemma}\label{lem:oddfirstwin}
Consider a disjunctive sum $g$ of an odd number of \bip strips. If $\Delta(g)=0$, then $g\cgfuzzy 0$.
\end{lemma}
\begin{proof}
Since $\Delta=0$, if one of the players does not have a \ubp, then neither does the other player. And similarly,  if one of the players plays a \ubp, then the other player can counter with another (perhaps on another strip). If Left plays a \ubp, to say $g'$, then $\Delta(g')=1$, and so there is at least one component with more white stones than black stones. In this component, there is a \ub for Right. 

In the first case, if no player has a \ubp, then the sum is of the form $\bl\wh+\cdots +\bl\wh\equiv \cgstar$, and hence the first player wins. 

In the second case, if the first player plays a \ubp, then the second player will respond with a \ubp, by the two-ahead-rule, and if the number of components stays the same, the result follows by induction. Moreover, this couple of moves cannot decrease the number of components, because every \ub removes exactly one stone, and no \ub is possible from $\bl\wh$. 
\end{proof}

If the number of component strips is even, the question seems a bit harder. Some further classification is required. Example~\ref{ex:nonendbypass} provides an N-position, and on the other hand, it is easy to find P-positions, by the mimic strategy. Here, we prove that, if each strip is a larvae, i.e. each strip contains a \ub for exactly one of the players and $\Delta=0$, then the previous player wins. Here is an example, where the players have 7 stones each: $\bl\wh\wh+\bl\wh\wh+\bl\wh\wh+\bl\bl\bl\bl\wh\equiv 0$.  

\begin{proposition}\label{prop:even}
Suppose that $g$ is a sum of an even number of larvae. If $\Delta(g)=0$, then $g\equiv 0$. 
\end{proposition}
\begin{proof}
If a player does not play a \ubp, then they will worsen the $\Delta$-excess from their point of view. This follows by the restriction on the possible strips. By the two-ahead-rule, then they will lose. Therefore we may assume that they play a \ubp. If, at some point, this is not possible, then, by the conclusion that $\Delta$ alternates between 0 and say 1 (if Left started), each strip is now $\cgstar$. Since there is an even number of strips, the second player wins. 

If, at some point, the first player plays instead in a component of the form $\bl\wh=\cgstar$, then, by Lemma~\ref{lem:oddfirstwin}, they will lose anyway, since play in $\bl\wh$ does not change $\Delta$, but gives an odd number of components. 
\end{proof}

One can obtain narrow bounds, up to one unit, by adding $\star$, $\up$ and $\upstar$ to a disjunctive sum of \bip strips with Delta-excess 0. We collect all similar results in one place, and include the result of Lemma~\ref{lem:oddfirstwin} as the first part of item 1. The result of Proposition~\ref{prop:even} is included as the second part in item 4. 

\begin{theorem}\label{thm:AW0}
Consider a disjunctive sum of $n$ \bip strips, $g=\sum_{i=1}^{n} s_i$, with $\Delta(g)=0$. If $n$ is odd, then 
\begin{enumerate}
\item $g\cgfuzzy 0$, and in particular, if each single strip $s_i\equiv\cgstar$, then $g\equiv \star$,
\item $\downstar < g < \upstar$, and 
\item $\cgdown \cglfuz g\cglfuz \cgup$. 
\end{enumerate}
If $n$ is even then 
\begin{enumerate}
\item[4.] $g\cgfuzzy \star$, and in particular, if each single strip is a larvae, then $g\equiv 0$,
\item[5.] $\cgdown < g < \cgup$, and 
\item[6.] $\downstar \cglfuz g\cglfuz \upstar$. 
\end{enumerate}
\end{theorem}
\begin{proof}
Lemma~\ref{lem:oddfirstwin} shows that the first player wins $g$ if $n$ is odd. If each strip $s_i\equiv \cgstar$, obviously $g\equiv \cgstar$. For case 4, assume that $n$ is even. 
Then $g+\star$ has an odd number of components, and item 1 applies since we can take, for example, $\bl\wh\equiv \star$. The second part of item 4 follows by Proposition~\ref{prop:even}.

Next, let us prove item 2, i.e. that Left wins $g+\upstar$, in case of $n$ odd. If Left starts by playing a \ubp, then she wins by the two-ahead-rule. Therefore, suppose that there is no \ubp. This leads to the first case of Lemma~\ref{lem:AW01}, and since n is odd, then she can play in $\upstar$ to obtain an even number of ``$\cgstar$''s. If Right starts, then, by the two-ahead-rule, he must play a \ubp, or remove the $\upstar$. If he plays a \ubp, then the total atomic weight of the strips becomes $-1$. Therefore Left has a \ubp, and she wins by induction. If there is no \ub and he removes the $\upstar$, then, since there is an odd number of strips and $\Delta=0$, Left wins playing first, by Lemma~\ref{lem:oddfirstwin}. 

For item 3, we use item 2. If Left removes a $\star$ in the game $g+\up+*+*=g+\upstar+*$, then since Left wins playing second in the game $g+\upstar$, by item 2, Left wins playing first in the game $g+\up+*+* \equiv g+\up$. The proofs of items 5 and 6 follow similarly by adding a $\star$ to the respective games.
\end{proof}

\section{Impossible normal play game values}\label{sec:imposs}
A simple problem is often raised in the analysis of games, but is only rarely answered: \textit{Given a ruleset, provide a simplest value, if it exists, that does not occur as a disjunctive sum of games in the ruleset.} 

For example, two, as yet, unanswered questions are:  ``does $\{1\mid 0, \{0 \mid -1\}\}$ occur in \textsc{toppling dominoes}?''\cite{FNSW};  and  ``When \textsc{clobber} is restricted to being played on a grid, does $*4$ occur?''\cite{AGNW}. In contrast,
Santos et al. \cite{SC} recently demonstrated that {\sc portuguese konane} is universal, meaning that it contains any short game value, and in particular it attains all all-small game values. On the other hand, we are not aware of any {\em all-small universal} all-small ruleset.

Since only games of integer atomic weights appear in \bipass, but the atomic weight can be any game, then the game is far from all-small universal, and hence there is an all-small game value of minimal birthday that does not appear as a position of a disjunctive sum of \bip strips. We will show that $*2$ does not appear as a disjunctive sum of \bipass.

Let us start by studying a single strip of \bipass.

\begin{lemma}\label{lem:nodouble0}
Let $g$ be  a single strip of \bipass. If $0\in\gL\cap \gR$ then $g=\bl\wh$. 
\end{lemma}
\begin{proof}
Suppose $0\in\gL\cap \gR$. In this case both players have a move to 0, which, by Proposition~\ref{cor:P0}, is the empty strip.

Now, if there are two or more black pieces in $g$ then no Left move results in the empty strip. Similarly, there can only be one white piece in $g$. Hence $g=\bl\wh$.
\end{proof}

In particular, Lemma~\ref{lem:nodouble0} shows that $*2=\{0,*\mid 0,*\}$ and $\upstar=\{0,*\mid 0\}$ do not occur as a value of a single strip. 

\begin{proposition}\label{prop:gsingle}
Suppose that $g\equiv *2$ or $g\equiv \upstar$. Then $g$ is not a single strip of \bipass.
\end{proposition}

The next result considers any number of strips.

\begin{lemma}\label{lem:gh}
Consider two games $g$ and $h$ in a given ruleset, such that $h \equiv \cgstar+g$, and where the component $*$ does not appear in the canonical form of $g$. If $*2$ does not appear in its literal form in this ruleset, then $h\not\equiv *2$.
\end{lemma}
\begin{proof}
A ruleset is closed under disjunctive sum, and taking options (but not necessarily under taking conjugate). The assumption that $*2$ does not appear in its literal form in the ruleset is the base case for induction, on literal form game trees that appear in this ruleset. Suppose that no game of smaller rank than $h$ has game value $*2$. Then $h = \cgstar+g$ implies $h\not \equiv *2$, because otherwise $g\equiv *3$. But $g$ has no option of value $*2$.
\end{proof}

\begin{theorem} \textsc{bipass} contains no position of value $*2$.
\end{theorem}
\begin{proof}

Let $h$ be the smallest \bip position, in terms of literal form game tree, with $h\equiv *2$. By Proposition~\ref{prop:gsingle} the game is not a single strip.
 
 Note that $h\equiv\{0,*\mid 0,*\}$ and that the atomic-weight of each of $h$, $0$ and $*$ is zero; hence, in each case, $\Delta = 0$. Moreover, by Table~\ref{tab:values}, the literal form $*2$ does not appear in a disjunctive sum of \bipass. Then Lemma~\ref{lem:gh} implies that any game of the form $*+g$ differs from $*2$. 

Now suppose that $h\equiv *2$ is a disjunctive sum of strips none of which is equivalent to $*$. By Theorem \ref{thm:AW0}, a move to 0 must be to an even number of strips, and likewise the move to $*$ must be to an odd number of strips. Therefore, each player has a move that eliminates a strip. The only strip both players
 can eliminate is $\bl\wh$ but that is contrary to our assumption on $h$. Hence,  the strip Left eliminates is a black headed larvae and the strip Right eliminates is a white headed larvae.
 However, eliminating a larvae changes the atomic-weight (Lemma \ref{lem:oneunit}). Since $h$, $0$ and $*$ all have atomic-weight 0, there is no move that eliminates a strip and leaves the atomic-weight at zero. Therefore $h$ is not of this form either. Thus, there is no \bip position of value $*2$.
\end{proof}

The proof becomes very neat by using Theorem~\ref{thm:AW0}, but one can also prove the last part by using instead Theorem~\ref{thm:bipass}, as follows. Consider a position of value $*2$, and not of the form $\cgstar +g$. 
One option must be to $\cgstar$, which, by Theorem~\ref{thm:bipass} then forces a position of the form $\bl\wh^n\bl\wh\bl^n\wh+g$, with $\Delta(g)=0$, since a \ub is not possible. This construction implies that $g\equiv 0$. But $\bl\wh^n\bl\wh\bl^n\wh\not\equiv *2$, because in this single strip there is no move to 0.

\section{Mis\`ere play}\label{sec:mis}
We prove a ``two-ahead rule'' for misere play {\sc bipass}. Let us begin with an example. We are not aware of any general such theory in mis\`ere play. 
\begin{example}\label{ex:mis}
Study the game $s=\bl\wh\wh\wh$ in mis\`ere play. Then $\Delta(s)=2$. If Left starts, she wins by playing to $\bl\wh$. If Right starts, then he must move to $\bl\wh\wh$, whereupon Left responds to $\bl\wh$, and wins again. This Left strategy obviously holds for any single strip of the form $\bl\wh^{n}$, if $n>2$ (but it fails when $n=2$ and Right starts).
\end{example}
The observation in Example~\ref{ex:mis} generalizes to play on arbitrary strips in a disjunctive sum as long as the total $\Delta$-excess $\ge 2$. 

Mis\`ere dicot play behaves better than the general class. In particular, we have the following useful analogue to normal-play. In mis\`ere play as in normal play, the game in which both players have a single move to end the game, is called `star' i.e. $*=\{0\mid 0\}$.
\begin{lemma}[\cite{MA}]\label{lem:allen}
Consider dicot mis\`ere play. Then $\star+\star\equiv 0$.
\end{lemma}
\begin{theorem}
Consider a disjunctive sum of {\sc bipass} strips $g$ under mis\`ere play. If $\Delta(g)\ge 2$, then Left wins. 
\end{theorem}
\begin{proof}
Let $g=\sum_{i=1}^k s_i$. We begin by proving that Left wins if Left starts. If there exists an index $i$ such that $s_i=\bl\wh$, then Left plays in $\bl\wh$ and wins by induction. Otherwise, if Left has a \ubp, then Left plays it,  giving $\Delta(g^L)=\Delta(g)+1\ge 2$. Again, Left wins by induction. 

If Left has no \ubp, and no strip is of the form $\bl\wh$, then $g=\sum \bl^{n_i}\wh$, with each $n_i\ge 2$. If $\Delta(g)>2$, then $\Delta(g^L)\ge 2$ and Left wins by induction. If $\Delta(g)=2$, then $g=\bl\wh\wh\wh\;$ or $g=\bl\wh\wh+\bl\wh\wh$, and, either way, Left wins going first. 

Finally, we prove that Left wins if Right is going first. If Right does not make a \ub then Left wins by induction. Suppose Right plays a \ubp. If $\Delta(g^R)\ge 2$, then Left wins by induction. So suppose $\Delta(g^R)=1$. If Left has a \ub, then Left wins by induction. Otherwise, by Lemma~\ref{lem:AW01} (which holds in mis\`ere play as well), $g^R=(\!\bl\wh)^{k-1}+\bl\wh\wh$. But, by Lemma~\ref{lem:allen} (\bip is a dicot ruleset), we have $\star+\star\equiv 0$, and so $g^R= \bl\wh+\bl\wh\wh\;$ or $g^R=\bl\wh\wh$. In both cases Left wins by moving to $\bl\wh$.
\end{proof}
In analogy with impartial theory, we may deduce that \bip is in essence a `tame' game: its behaviour is not too different from that for normal play when the atomic-weight is other than $-1$, $0$ or $1$. 

\section{Further reflections}
How have atomic weights proven useful in other games? In \textsc{clobber} they were used to show that the game is NP-hard. They are easy to calculate in \textsc{cutthroat stars} and \textsc{yellow-brown hackenbush} and are integers. In these two games, atomic-weights are unnecessary since a complete solution is known. However, as \bip exemplifies, integer atomic-weights does not assure that the values are easy to calculate. For every non-empty position in \textsc{hackenbush sprigs} that is not covered by the atomic-weight two-ahead rule,  the value of an extra parameter determines the outcome.  (See \cite{C} for the normal-play version, and \cite{MMN} considers the mis\`ere version.) Nothing is known about the atomic weights of \textsc{partizan euclid} although the plot of the mean-values of the atomic weights of the first 10,000 positions is intriguing. 

\bip originated in the Games-at-Dal workshop, 2012 as a game on Ferrer's Diagrams: Left is allowed to remove a portion of a row and Right a portion of a column, provided what remains is a Ferrer's diagram. 
In Figure~\ref{fig:ferdia}, the starting position in Example~\ref{ex:singleline} is encoded as a Ferrer's Diagram.
\begin{figure}[ht]
\begin{center}
\begin{tikzpicture}[scale=0.65, line width=1pt, rotate=90]
  \draw (0,0) grid (4,1);
  \draw (0,0) grid (1,-1);
  \node[left] (A) at (4.23,0.15) {$\bl$};
  \node[left] (A) at (3.5,-.58) {$\wh$};
  \node[left] (A) at (2.5,-.58) {$\wh$};
  \node[left] (A) at (1.6,-.58) {$\wh$};
  \node[left] (A) at (1.25,-.95) {$\bl$};
  \node[left] (A) at (.5,-1.58) {$\wh$};
\end{tikzpicture}\caption{{\sc bipass} has an equivalent Ferrer's diagram interpretation.}\label{fig:ferdia}
\end{center}
\end{figure}
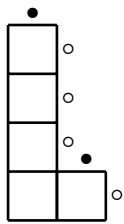
In Figure~\ref{fig:ferplay}, we play out the game from Example~\ref{ex:singleline}, using this encoding.

\begin{figure}[ht!]
\begin{center}
\begin{tikzpicture}[scale=0.65, line width=1pt, rotate=90]
  \draw (0,0) grid (4,1);
  \draw (0,0) grid (1,-1);
  \node[left] (A) at (1,-2) {$\rightarrow$};
\end{tikzpicture}
\begin{tikzpicture}[scale=0.65, line width=1pt, rotate=90]
  \draw (0,0) grid (2,1);
  \draw (0,0) grid (1,-1);
\node[left] (A) at (1,-2) {$\rightarrow$};
\end{tikzpicture}
\begin{tikzpicture}[scale=0.65, line width=1pt, rotate=90]
  \draw (0,0) grid (2,1);
  \node[left] (A) at (.5,-1) {$\rightarrow$};
\end{tikzpicture}
\begin{tikzpicture}[scale=0.65, line width=1pt, rotate=90]
\draw (0,0) grid (0,1);
\end{tikzpicture}\caption{In this Ferrer's Diagram interpretation of {\sc bipass}, Left begins by removing two pieces of the first column, and then Right plays by removing one piece of the first row. At last Left removes the first column.}\label{fig:ferplay}
\end{center}
\end{figure}
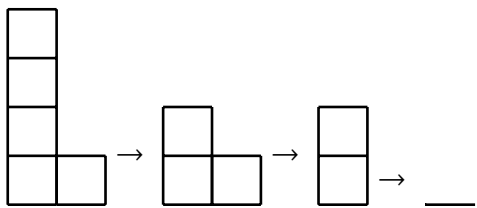

The impartial \textsc{welter's game} is played on a strip with stones, and it is equivalent to removing hooks from a Ferrer's Diagram. The analysis involves a beautiful application of frieze patterns. See \cite[Volume 3, Chapter 15]{WW}. 

For both games (and other Ferrer's diagram games) the translations to a stones-on-a-strip is obtained by tracing out the envelope of the diagram putting a black stone on a horizontal unit line and a white stone on a vertical unit line. See Figure~\ref{fig:ferdia}. (In the case of the \textsc{welter's game} and other impartial games, the white stones become the empty spaces.) 

As anticipated by Section~\ref{sec:imposs}, we suggest the following problem.

\begin{problem}
Find an all-small ruleset that is all-small universal.
\end{problem}

Note that integer an atomic weight of a game does not necessarily imply that the options have integer atomic weights. Take $\aw (\gL)=\aw(\gR)=1/2$. Then, with notation as in Theorem~\ref{thm:AWcon}, $G=0$, and integer. 
The all-small games with integer atomic weights, and with followers of integer atomic weights, form a subgroup of the all-small games, say $A$. (This subgroup is a universe of games, although it is not parental, as defined in \cite{abs}.)

\begin{problem}\label{prob:2}
Find a universal ruleset in $A$.
\end{problem}
Since $\aw(*2)=\aw(*)=\aw(0)=0$, an integer, and $*2$ does not occur in \bipass, this ruleset is not universal in the sense of Problem \ref{prob:2}.

The ruleset {\sc maximal bipass} is as \bipass, but where every move must be the longest jump possible by the piece. 

\begin{problem}\label{prob:3}
Are the values of {\sc maximal bipass} disjunctive sums of star-based ordinal sum \cite{NM}? That is, for any game $g$ in {\sc maximal bipass}, 
are there games $g_1,g_2,\ldots,g_k$  such that $g=*:g_1+*:g_2 +\ldots +*:g_k$? (This is true for games of the form $\wh^m\bl^n$, $\wh^1\bl^n\wh^p\bl^q$.)
\end{problem}

The game of {\sc cannibal} \bip is as \bipass,\footnote{Inspried by the game {\sc cannibal clobber} \cite{Alt}.} but where an amoebae may instead of bypassing members of the other tribe, clobber (or eat) any number of neighbors of their own kind, playing the same direction as in \bipass. Thus, for example $\bl\bl\wh=\cg{\bl\wh,\bl\wh\bl}{\bl\wh\bl,\wh\bl\bl\,}$, $\wh\wh\bl\bl=\cg{\wh\wh+\bl}{\wh+\bl\bl\,}$, $\bl\bl\bl=\cg{\bl\bl,\bl+\bl,0}{\varnothing\,}$ and $\bl\wh\bl\bl=\cg{\wh\bl\bl\bl, \bl\wh+\bl}{\wh\bl\bl\bl\,}$. Single stones cannot be used, so for example $\bl +\bl =0$, and a single white stone to the left, where all other stones are black, cannot be used, so that $\wh\bl\bl\bl=\bl\bl\bl$. Some properties are immediate. A single strip may now decompose to a disjunctive sum of strips, and pieces remain alive, until the end of play, with a few exeptions as noted. And moreover the game is no more all-small, as for example $\bl\bl\wh \equiv \bl\bl =1$. For standard \bipass, it is beneficial to have fewer pieces. But now a large number of own pieces is beneficial. This variation contains non-trivial one-strip P-positions, such as $\bl\bl\wh\wh$ and $\bl\wh\wh\bl$. 
\begin{problem}\label{prob:4}
Is {\sc cannibal} \bip tepid, i.e. are all positions of the form a number + an infinitesimal? More precisely, for a disjunctive sum of games, $g$, is the value $\Delta(g)+\epsilon$? 
\end{problem}
If the value is $\Delta+\epsilon$, then the winner whenever $\Delta=0$ is determined by \bip play, because no player wants to start eating their own pieces.

\end{document}